\documentclass{amsart}
\usepackage[english]{babel}
\usepackage{graphicx}
\usepackage{amsmath}
\usepackage{amsthm}
\usepackage{amssymb}
\newtheorem{Lemma 1}{Lemma}[section]
\newtheorem{Theorem}[Lemma 1]{Theorem}
\newtheorem{Lemma 3}[Lemma 1]{Lemma}
\newtheorem{Lucas The}[Lemma 1]{Theorem}
\newtheorem{num}[Lemma 1]{Lemma}
\newtheorem{lem}[Lemma 1]{Lemma}
\newtheorem{lemme}[Lemma 1]{Lemma}
\newtheorem{lem 1}[Lemma 1]{Lemma}
\newtheorem{Main}[Lemma 1]{Theorem}
\newtheorem{Corollary}[Lemma 1]{Corollary}
\newtheorem{Corollary 1}[Lemma 1]{Corollary}
\newtheorem{Corollary 2}[Lemma 1]{Corollary}

\begin{document}
\title{On Integer sequences in Product sets}
\address{Department of Mathematics, Indian Institute of Technology Roorkee,India 247667}
\author{Sai Teja Somu}
\date{October 29 , 2015}
\maketitle
\begin{abstract}
Let $B$ be a finite set of natural numbers or complex numbers. Product set corresponding to $B$ is defined by $B.B:=\{ab:a,b\in B\}$. In this paper we give an upper bound for longest length of consecutive terms of a polynomial sequence present in a product set accurate up to a positive constant. We give a sharp bound on the maximum number of Fibonacci numbers present in a product set when $B$ is a set of natural numbers and a bound which is accurate up to a positive constant when $B$ is a set of complex numbers. 
\end{abstract}

\section{Introduction}
In \cite{A} and \cite{B} the author has proved that if $B$ is a set of natural numbers then the product set corresponding to $B$ cannot contain long arithmetic progressions. In \cite{A} it was shown that the longest length of arithmetic progression is at most $O(|B|\log |B|)$. We try to generalize this result for polynomial sequences. Let $P(x)\in \mathbb{Z}[x]$ be a polynomial with positive leading coefficient. Let $R$ be the longest length of consecutive terms of the sequence, that is, \[R=max\{n : \text{there exists an $x\in \mathbb{N}$ such that }  \{P(x+1),\cdots,P(x+n)\}\subset B.B\}.\]
We prove that $R$ cannot be large for every polynomial $P(x)$. In section 2 we consider the question of determining maximum number of Fibonacci and Lucas sequence terms in a product set. 

As in \cite{A} we define an auxiliary bipartite graph $G(A, B.B)$ and auxiliary graph $G'(A,B.B)$ which are  constructed for any sets $A$ and $B$ whenever $A\subset B.B$. The color classes of $G$ are two copies of $B$  whereas $G'$ has only one color copy of $B$ and for each $a\in A$
we pick a unique representation $a = b_1b_2$ and place an edge $(b_1, b_2)$ in $G$ and in $G'$. Note that $V(G) = 2|B|$, $V(G')=|B|$ and
$E(G)=E(G')=|A|$. Observe that $G'$ can have self loops and $G$ cannot have self loops.
\section{Number of Fibonacci Numbers and Lucas Numbers}
Let $B$ be a finite set of naural numbers. Let $A$ be the set of Fibonacci numbers contained in the product set. From \cite{C} there are only two perfect square Fibonacci numbers, viz., $1$ and $144$. Hence there can be at most two self loops in the graph $G'(A,B.B)$. We give an upper bound on cardinality of $A$ by using the following lemma.

\begin{Lemma 1}\label{Lemma}
Let $F_n$ and $F_m$ be $n$th and $m$th Fibonacci numbers and $m<n$ and $n>2$ then $gcd(F_n,F_m)<\sqrt{F_n}$.
\end{Lemma 1}
\begin{proof}
Let $d=gcd(m,n)$. From the strong divisibilty property of Fibonacci numbers $gcd(F_n,F_m)=F_{d}$.
 We know that $F_n=\frac{\alpha^n-\beta^n}{\alpha-\beta},$ where $\alpha=\frac{1+\sqrt{5}}{2}$ and $\beta=\frac{1-\sqrt{5}}{2}$. Since $m<n$, clearly $d\leq \frac{n}{2}$. 
 If $d=1$ then the hypothesis is clearly true and if $d>1$, we have
\begin{equation*}
(F_d)^2=\frac{(\alpha^d-\beta^d)^2}{(\alpha-\beta)^2}< \frac{(\alpha^{2d}-\beta^{2d})}{(\alpha-\beta)}\leq F_n.
\end{equation*}
Thus, $gcd(F_m,F_n)<\sqrt{F_n}$.
\end{proof}

\begin{Theorem}
There cannot be more than $|B|$ Fibonacci numbers in the product set $B.B$ when $B$ is a set of natural numbers.
\end{Theorem}
\begin{proof}
We claim that in the graph $G'(A,B.B)$ there cannot any cycle other than self loops. Suppose there is a $k$-cycle $b_1b_2\cdots b_kb_1$ which implies that $b_ib_{i+1}$ for $1\leq i \leq k-1$ and $b_kb_1$ are distinct Fibonacci numbers in the set $B.B$. Without loss of generality let us assume $b_1b_2=F$ is the largest Fibonacci number among $b_ib_{i+1}$ for $1\leq i \leq k-1$ and $b_kb_1$.
From Lemma \ref{Lemma}, we have 
\begin{align*}
& b_1\leq gcd(b_1b_2,b_1b_k)<\sqrt{F},\\
& b_2\leq gcd(b_1b_2,b_2b_3)<\sqrt{F}.
\end{align*}
Hence $F=b_1b_2<F$ which is a contradiction.
Hence there cannot be any cycle. From \cite{C} there cannot be more than 2 self loops.
Hence the number of edges which equal number of Fibonacci numbers in the set $B.B$ cannot exceed $|B|+1$.

Now we prove that there cannot be $|B|+1$ Fibonacci numbers. Suppose there are $|B|+1$ Fibonacci numbers, as the graph cannot have any cycle there should be two self loops namely, $1$ and $12$ and the graph obtained by removing the two self loops should be connected tree of $|B|$ vertices. Since the graph is connected there should be a path between $1$ and $12$. Let the path be $b_1b_2\cdots b_k$ which implies that  $b_ib_{i+1}$ for $1\leq i\leq k-1$ are Fibonacci numbers and without loss of generality assume $b_1=1$ and $b_k=12$. Let $l$ be the index of highest value of $b_ib_{i+1}$. Clearly $l\neq 1$ and if $2\leq l\leq k-2$ then from Lemma \ref{Lemma}  
\begin{align*}
& b_l\leq gcd(b_lb_{l+1},b_{l-1}b_l)<\sqrt{b_lb_{l+1}},\\
& b_{l+1}\leq gcd(b_lb_{l+1},b_{l+1}b_{l+2})<\sqrt{b_lb_{l+1}}.
\end{align*}
Which implies $b_lb_{l+1}<b_lb_{l+1}$. Hence $l=k-1$. Again from Lemma\ref{Lemma}
\begin{equation*}
b_{k-1}\leq gcd(b_{k-1}b_k,b_{k-2}b_{k-1})<\sqrt{b_{k-1}b_k}
\end{equation*}
which implies $b_{k_1}<b_k=12$ but there are no Fibonacci numbers of the form $12b$ with $b<12$.
Hence there cannot be $|B|+1$ Fibonacci numbers. Thus number of Fibonacci numbers in the set $B.B$ is $\leq |B|$. 
\end{proof}
Now we consider the case where $B$ is a set of complex numbers and try to give an upper bound on the number of Lucas sequence terms in the product set. Let $A$ be the set of Lucas sequence terms with indices greater than $30$ in the  product set $B.B$.   

\begin{Lemma 3}\label{Lemma 3}
There cannot be any cycle in $G(A,B.B)$.
\end{Lemma 3}
\begin{proof}
Suppose there was a cycle $b_1b_2\cdots b_{2k}b_1$.
Then 
\begin{align*}
  b_1b_2&=L_{n_1}\\
  b_2b_3&=L_{n_2}\\
        &.       \\
        &.\\
        &.\\
  b_{2k}b_1&=L_{n_{2k}},      
\end{align*}
where $L_{n_i}$ are Lucas sequence terms with indices greater than $30$, which implies
\begin{equation}
\prod_{i=1}^{k}L_{n_{2i}}=\prod_{j=1}^{k}L_{n_{2j-1}}
\end{equation}
Let $n_i$ be the largest index $\geq 31$. Then from \cite{D} $L_{n_i}$ contains a primitive divisor $p$ and hence $p$ divides exactly one side of (1) and therefore (1) cannot be true. Thus there cannot be any cycle.
\end{proof}

\begin{Lucas The}
Let $(L_n)_{n=1}^\infty$ be a Lucas sequence. Then number of distinct elements of $(L_n)_{n=1}^\infty$ in $B.B$ is less than $2|B|+30$.
\end{Lucas The}
\begin{proof}
Since number of vertices in $G(A,B.B)$ is $2|B|$ and from Lemma \ref{Lemma 3} there cannot be a cycle in $G(A,B.B)$ the number of edges in $G(A,B.B)$ $\leq 2|B|-1$. Hence the number of distinct terms in Lucas sequence of index $\geq 31$ is $\leq 2|B|-1$. Hence number of distinct Lucas sequence terms in $B.B$ is $\leq 2|B|+29$.
\end{proof}
    
\section{Polynomial sequences}

Now we turn onto the second problem in this paper. Given a polynomial $P(x)$ with positive leading coefficient and integer coefficients what can we say about the longest length of consecutive terms in the product set $B.B$.

 Since there can be at most finitely $r$ such that $P(r)\leq0$ or $P'(r)\leq 0$ there exists an $l$ such that $P(r+l)>0$ and $P'(r+l)>0$ for all $r\geq 1$. Hence we can assume without loss of generality that every irreducible factor $f(x)$ of $P(x)$ $f(x)>0 \text{ and } f'(x)>0~~ \forall x\geq 1$ as this assumption only effects $R$  by a constant. From now  we will be assuming that for every irreducible divisor $f(x)$ of $P(x)$ $f(x)>0$ and $f'(x)>0$ for all natural numbers $x$. 
   We prove three lemmas in order to give an upper bound for $R$.

 If $f(x)\in \mathbb{Z}[x]$ is an irreducible polynomial of degree $\geq 2$. Let $D$ be the discriminant of $f(x)$. Let $d$ be the greatest common divisor of the set $\{f(n): n\in \mathbb{N}\}$. Let $f_1(x)=\frac{f(x)}{d}.$ Denote $Dd^2$ by $M$. If $p^e||M$ then $p^e\nmid d$ and hence there exists an $a_p$, such that $f_1(x)$ is not divisible by $p$ for all $x\equiv a_p (\mod p^e)$. 
  From Chinese remainder theorem there exists an integer $a$ such that $a\equiv a_p(\mod p^e)$ for all primes dividing $M$ and hence there exists an $a$ such that $f_1(x)$ is relatively prime to $M$ for all $x\equiv a (\mod M)$. 

\begin{num}\label{Lemma 4} 
For sufficiently large $R$
the number of numbers in the set $\{f_1(r+i) : 1\leq i \leq R,r+i\equiv a \mod M \}$ with atleast one prime factor greater than $R$ is $\geq \frac{R}{3M}$ for every non negative integer $r$. 
\end{num}
\begin{proof}
Let \[Q=\prod_{\substack{i=1\\r+i\equiv a\mod M}}^{R}f_1(r+i).\]
Let $S$ be the largest divisor of $Q$ with all prime factors $\leq R$.
Let $e_p$ be the index of $p$ in $S$. Let $\rho(p)$ denote the number of solutions modulo $p$ of the congruence $f(x)\equiv 0 (\mod p)$.
\begin{align*}
\log S&=\sum_{\substack{p\nmid M\\p\leq R}}e_p\log p \\
&=\sum_{\substack{p\nmid M\\p\leq R}}\sum_{n=1}^{\lfloor\frac{\log f(r+R)}{\log p}\rfloor}\sum_{\substack{i=1\\r+i\equiv a\mod M\\f_1(r+i)\equiv 0 \mod p^n}}^{R}\log p\\
&=\sum_{\substack{p\nmid M\\p\leq R}}\sum_{n=1}^{O(\frac{\log (r+R)}{\log p})}(\frac{\rho(p)\log p R}{Mp^n}+O(\log p))\\
&=\sum_{\substack{p\nmid M\\p\leq R}}\frac{\rho(p)\log p R}{Mp}+O\left(\frac{\log (r+R)R}{\log R}\right).
\end{align*}
From prime ideal theorem, we have
\[
\sum_{\substack{p\nmid M\\p\leq R}}\frac{\rho(p)\log p R}{Mp}=\frac{R\log R}{M}+O(R).
\]
Thus, we have
\begin{equation*}
\log S=\frac{R\log R}{M}+O\left(\frac{\log (r+R)R}{\log R}\right). 
\end{equation*}
Let $L$ be a subset of $\{f_1(r+i) : 1\leq i\leq R,r+i\equiv a \mod M\}$ containing all the numbers which do not contain any prime factor greater than $R$ and let $l$ denote the cardinality of $L$.
\begin{align*}
\log \prod_{\substack{i=1\\f_1(r+i)\in L}}^{R}f_1(r+i)&\geq \log \prod_{i=1}^{l}f_1(r+i)\\
&=n\sum_{i=1}^{l}\log (r+i)+O(l)\\
&=nl\log (r+l)+O(l),
\end{align*}
where $n$ is the degree of the polynomial $f(x)$.
Hence 
\begin{equation*}
nl\log (r+l)+O(l)\leq \frac{R\log R}{M}+O\left(\frac{\log (r+R)R}{\log R}\right).
\end{equation*}
Hence for sufficiently large $R$, $l$ should be less than $ \frac{2R}{3M}-2$. Hence number of numbers belonging to the set $\{f_1(r+i) : 1\leq i \leq R,r+i\equiv a \mod M\}$ with atleast one prime factor greater than $R$ is $\geq \frac{R}{3M}$.
\end{proof}
 
The following corollary immediately follows from Lemma \ref{Lemma 4}. 
\begin{Corollary}\label{Corollary}
If $P(x)$ has an irreducible divisor of degree $\geq 2.$ Then there exist $\Omega(R)$ numbers in the set $\{P(r+i) : 1\leq i \leq R\}$ with atleast one prime factor greater than $R$. 
\end{Corollary}

\begin{lem}\label{Lemma 5}
 If $f(x)$ is a linear polynomial. If $r\geq R^\gamma$ for a $\gamma >1$ then there exists a constant $c>0$ depending upon $\gamma$ such that for sufficiently large $R$, number of numbers of the set $\{f(r+i) : 1\leq i \leq R\}$ with a prime factor greater than $R$ is greater than $cR$.
\end{lem}
\begin{proof}
The proof is similar to that of Lemma \ref{Lemma 4}.
Let $Q=\prod_{i=1}^{R}f(r+i)$
and $S$ be the largest divisor of $Q$ with all prime factors  $\leq R$.
\begin{align*}
\log S&=\sum_{p\leq R}e_p\log p\\
&=\sum_{p\leq R}\sum_{n=1}^{O\left(\frac{\log (r+R)}{\log p}\right)}\sum_{\substack{i=1\\f(r+i)\equiv 0 \mod p^n}}^{R}\log p\\
&=\sum_{p\leq R}\frac{R\log p}{p}+O\left(\frac{R\log (r+R)}{\log R}\right)\\
&=R\log R+O\left(\frac{R\log (r+R)}{\log R}\right).
\end{align*}
Let $L$ be a subset of \{$1\leq i\leq R$\} containing all $i$ such that $f(r+i)$ has all prime factors $\leq R$. Let the cardinality of $L$ be $l$.
\begin{align*}
\log \prod_{\substack{i=1\\i\in L}}^{R}f(r+i)&\geq \log \prod_{i=1}^{l}f(r+i)\\
&=l\log (r+R) +O(R). 
\end{align*}
which implies \begin{equation*}
 l \log (r+R)+O(R)\leq R\log R +O\left(\frac{R\log(r+R)}{\log R}\right).
\end{equation*}
For sufficiently large $R$, $l$ should be $\leq \frac{(1+\gamma)}{2\gamma}R$. Hence for sufficiently large $R$ number of numbers of the set $\{f(r+i) : 1\leq i \leq R  \}$ with atleast one prime factor greater than $R$ is $\geq \frac{(\gamma-1)R}{2\gamma }$.
\end{proof}
We have the following Corollary for Lemma \ref{Lemma 5}.
\begin{Corollary 1}\label{Corollary 2}
If degree of every irreducible divisor of $P(x)$ is 1 and $r\geq R^\gamma$ then number of elements of the set  $\{P(r+i): 1\leq i \leq R\}$ with atleast one prime factor greater than $R$ is $\Omega(R)$.
\end{Corollary 1}

\begin{lemme}\label{Lemma 6}
Let $f(x)$ be a linear polynomial.
If $r\leq R^\gamma$ for some $\gamma>1$ then there are $\Omega(\frac{R}{\log R})$ numbers of the set $\{f(r+i): 1\leq i \leq R \}$ with atleast one prime factor greater than $\frac{R}{2}$. 
\end{lemme}
\begin{proof}
Let $f(n)=an+b$ then there are $\Omega (\frac{R}{\log R})$ primes between $(\frac{R}{2},R]$ which are coprime to $a$. Each prime has one or two $i\in [1,R]$ such that $p|f(r+i)$. For each $f(r+i)$ there are at most $O(1)$ prime divisors belonging to $(\frac{R}{2},R]$. Hence there are $\Omega(\frac{R}{\log R})$ numbers with atleast one prime factor greater than $\frac{R}{2}$.
\end{proof}

\begin{Corollary 2}\label{Corollary 3}
If degree of every irreducible divisor of $P(x)$ is 1 and $r\leq R^\gamma$ then number of elements of the set  $\{P(r+i): 1\leq i \leq R\}$ with atleast one prime factor belonging to the range $(\frac{R}{2},R]$ is $\Omega(\frac{R}{\log R})$.
\end{Corollary 2}

In a graph $G(V,E)$ for $v\in V$ we define $V(v)$ to be the set of all vertices adjacent to $v$.

\begin{lem 1}\label{Lemma 7}
If there is a bipartite graph $(A,B,E)$ such that for all $a\in A$ and $b\in B$, degree of $a$ is $\leq n$ and degree of $b$ is $\geq 1$ then there exists a sequence of vertices $b_1,\cdots,b_k$ with $b_i\in B$ satisfying $V(b_1)\neq \phi$ and $V(b_i)/(\cup_{j=1}^{j=i-1}V(b_j))\neq \phi$ for $2\leq i\leq k$ and $k\geq \frac{|B|}{n}$.
\end{lem 1}
\begin{proof}
The proof is by induction on $n$. For $n=1$ the lemma is true since degree of $a\leq 1~~ \forall ~ a\in A\implies V(b_1)\cap V(b_2)=\phi ~~\forall~ b_1\neq b_2\in B$ and the sequence $b_1,\cdots, b_{|B|}$ will clearly satisfy $V(b_1)\neq \phi$ and $V(b_i)/(\cup_{j=1}^{i-1}V(b_j))\neq \phi$ for  $2\leq i\leq k$.
If the lemma is true for $n=r$ we have to prove for $n=r+1$. Order the vertices of $B$ as $b_1,\cdots b_{|B|}$. Let 
$S=\{a\in A : \text{degree of $a$}\geq 1\}$.
Let $S_1=V(b_1)$ and for $2\leq i \leq |B|$, let  $S_i=V(b_i)/(\cup_{j=1}^{i-1}V(b_j))$.
Observe that $S=\cup_{i=1}^{|B|}S_i$.
Let $K$ be a set defined by $K=\{b_i : S_i\neq \phi\}$.
If $|K|\geq \frac{|B|}{r+1}$ then we can choose the vertices in the set $K$ arranged in a sequence which satisfies the hypothesis.
If $|K|<\frac{|B|}{r+1}$ then consider the induced subgraph $A\cup (B-K)$ then degree of $a$ is less than or equal to $r$ for all $a\in A.$ From the induction assumption there exists a sequence with length $\geq \frac{|B-K|}{r}>|B|(1-\frac{1}{r+1})\frac{1}{r}=\frac{|B|}{r+1}$ in $B-K$ satisfying the hypothesis which completes the proof by induction.   
\end{proof}

\begin{Main}
Let $P\in Z[x]$ and has a positive leading coefficient and if $\{P(r+1),\cdots,P(r+R)\}$ is contained in the product set $B.B$ for a nonnegative integer $r$ and natural number $R$ and $B$ is a set of complex numbers. \newline
(1)If $P$ has an irreducible factor of degree $\geq 2$ then $R=O(|B|)$.\newline
(2)If $P$ has no irreducible factor of degree $\geq 2$ and $r>R^\gamma$ and $\gamma >1$ then $R=O(|B|)$.\newline
(3)If $P$ has no irreducible factor of degree $\geq 2$ and $r\leq R^\gamma$ and $\gamma >1$ then $R=O(|B|\log |B|)$.
\end{Main}
\begin{proof}
If $P$ has an irreducible factor $f$ of degree greater than $2$ or $P(x)$ has no irreduible divisor of degree $\geq 2$ and $r> R^\gamma$ 
let 
\begin{equation*}
A=\{p : \text{$p$ is a prime,~$p|P(r+i)$ for some $1\leq i\leq R$,~$p>R$}\}
\end{equation*}
and let
\begin{equation*}
C=\{P(r+i) : \text{$1\leq i \leq R$,$\exists$ prime $p>R$ \text{ such that }  $p|P(r+i)$}\}.
\end{equation*}
If $P(x)$ has no irreducible divisor of degree $\geq 2$ and $r\leq R^\gamma$ then let 
\begin{equation*}
A=\{p : \text{$p$ is a prime, $\frac{R}{2}< p\leq R$ and 
$p|P(r+i)$ for some $1\leq i \leq R$}\}
\end{equation*}
and let 
\begin{equation*}
C=\{P(r+i) : \text{$1\leq i \leq R$,$\exists$ prime $p\in (\frac{R}{2},R]$ \text{ such that } $p|P(r+i)$}\}.
\end{equation*}
In cases (1) and (2) from Corollaries \ref{Corollary} and \ref{Corollary 2} the size of $C$ is $\Omega(R)$. In case (3) from Corollary \ref{Corollary 3} the size of $C$ is $\Omega(\frac{R}{\log R})$.
If we consider a bipartite graph $G$ between $A\cup C$ constructed such that there exits an edge $p\in A$ and $P(r+i)\in C$ if and only if $p|P(r+i)$. In this graph the degree of $a\in A$ is is less than or equal to the degree of polynomial $P$. Hence from Lemma \ref{Lemma 7} there exists a sequence $c_1,c_2,\cdots ,c_k$ with $k\geq \frac{|C|}{\text{degree of $P$}}$ such that $V(c_1)\neq \phi$ and $V(c_i)/\cup_{j=1}^{k-1}V(c_j)\neq \phi$.  Therefore every $c_i$ has a prime divisor which does not divide any of $c_j$ for $1\leq j \leq i-1$. Let  $C'=\{c_1,\cdots,c_k\}$. Note that in cases (1) and (2) $|C'|=\Omega(R)$  and in case 3 $|C'|=\Omega(\frac{R}{\log R})$.
\par
Consider the bipartite auxiliary graph $G(C',B.B)$. We claim that there cannot be any cycle in this graph. Suppose there was a cycle $b_1b_2\cdots b_{2k}b_1$ then 
\begin{align*}
  b_1b_2&=c_{n_1}\\
  b_2b_3&=c_{n_2}\\
        &.       \\
        &.\\
        &.\\
  b_{2k}b_1&=c_{n_{2k}}    
\end{align*}
and 
\begin{equation}
\prod_{i=1}^{k}c_{n_{2i}}=\prod_{j=1}^{k}c_{n_{2j-1}}
\end{equation}
let $n_i$ be the highest index present in the cycle. There exists a prime $p$ such that $p|c_{n_i}$ and $p\nmid c_{n_j}$ for $j\neq i$ and hence $p$ divides exactly one side of (2) and hence (2) cannot be true. Thus there exists no cycle in $G(C',B.B)$. Hence $|C'|\leq 2|B|-1$. Therefore $R=O(|B|)$ in cases (1),(2) and $R=O(|B|\log |B|)$ in case (3) which completes the proof of the theorem. 
\end{proof}
\bibliographystyle{amsplain}

\end{document}